\documentclass{amsart}
\usepackage{amssymb}
\usepackage{amsmath}

\newtheorem{theorem}{Theorem}[section]

\newtheorem{corollary}[theorem]{Corollary}

\newtheorem{lemma}[theorem]{Lemma}

\newtheorem{problem}[theorem]{Problem}

\theoremstyle{definition}

\newtheorem{example}[theorem]{Example}
\newtheorem{remark}[theorem]{Remark}

\newcommand{\MC}{\mathcal{MC}}

\begin{document}
\author{Taras Banakh}
\address{\emph{T.Banakh}: Department of Mathematics, Ivan Franko National University of Lviv, Universytetska 1, 79000 Lviv, Ukraine and Institute of Mathematics,
and
 Jan Kochanowski University in Kielce, \'Swi\c etokrzyska 15, 25-406 Kielce, Poland}
\author{Artur Bartoszewicz}
\address{\emph{A.Bartoszewicz, E.Szymonik}: Institute of Mathematics,
Technical University of \L \'od\'z, W\'olcza\'nska 215, 93-005 \L \'od\'z,
Poland}
\author{Ma\l gorzata Filipczak}
\address{\emph{M.Filipczak}: Faculty of Mathematics and Computer Sciences,
\L \'od\'z University, Banacha 22, 90-238 \L \'od\'z, Poland}
\author{Emilia Szymonik}

\title[]{Topological and measure properties of some self-similar sets}

\date{}

\thanks{The first author has been partially financed by NCN grant
DEC-2012/07/D/ST1/02087.}

\begin{abstract}
Given a finite subset $\Sigma\subset\mathbb{R}$ and a positive real number $%
q<1$ we study topological and measure-theoretic properties of the
self-similar set $K(\Sigma;q)=\big\{\sum_{n=0}^\infty
a_nq^n:(a_n)_{n\in\omega}\in\Sigma^\omega\big\}$, which is the unique
compact solution of the equation $K=\Sigma+qK$. The obtained results are
applied to studying partial sumsets $E(x)=\big\{\sum_{n=0}^\infty
x_n\varepsilon_n:(\varepsilon_n)_{n\in\omega}\in\{0,1\}^\omega\big\}$ of
some (multigeometric) sequences $x=(x_n)_{n\in\omega}$.
\end{abstract}

\maketitle

\section{Introduction}

Suppose that $x=\left( x_{n}\right) _{n=1}^{\infty }$ is an absolutely
summable sequence with infinitely many nonzero terms and let
\begin{equation*}
E( x) =\Big\{\sum_{n=1}^\infty\varepsilon
_{n}x_{n}:(\varepsilon_{n})_{n=1}^{\infty }\in \{ 0,1\}^{\mathbb{N}}\Big\}
\end{equation*}%
denote the set of all subsums of the series $\sum_{n=1}^{\infty }x_{n},$
called \emph{the achievement set} (or \emph{a partial sumset}) of $x$. The
investigation of topological properties of achievement sets was initiated
almost one hundred years ago.\ In 1914 Soichi Kakeya \cite{K} presented the
following result:

\begin{theorem}[Kakeya]
\label{kakeya} For any sequence $x\in l_{1}\setminus c_{00}$

\begin{enumerate}
\item $E(x)$ is a perfect compact set.

\item If $|x_{n}|>\sum_{i>n}|x_{i}|$ for almost all $n$, then $E(x)$ is
homeomorphic to the ternary Cantor set.

\item If $|x_{n}|\leq \sum_{i>n}|x_{i}|$ for almost all $n$, then $E(x)$ is
a finite union of closed intervals. In the case of non-increasing sequence $%
x $, the last inequality is also necessary for $E(x)$ to be a finite union
of intervals.
\end{enumerate}
\end{theorem}

Moreover, Kakeya conjectured that $E(x)$ is either nowhere dense or a finite
union of intervals. Probably, the first counterexample to this conjecture
was given by Weinstein and Shapiro (\cite{WS}) and, independently, by Ferens
(\cite{F}). The simplest example was presented by Guthrie and Nymann \cite%
{GN}: for the sequence $c=\big(\frac{5+(-1)^n}{4^n}\big)_{n=1}^\infty$, the
set $T=E(c)$ contains an interval but is not a finite union of intervals. In
the same paper they formulated the following theorem, finally proved in \cite%
{NS2}:

\begin{theorem}
\label{Guthrie Nymann} For any sequence $x\in l_{1}\setminus c_{00}$, $E(x)$
is one of the following sets:

\begin{enumerate}
\item a finite union of closed intervals;

\item homeomorphic to the Cantor set;

\item homeomorphic to the set $T$.
\end{enumerate}
\end{theorem}

Note, that the set $T=E(c)$ is homeomorphic to $C\cup \bigcup_{n=1}^{\infty
}S_{2n-1}$, where $S_{n}$ denotes the union of the $2^{n-1}$ open middle
thirds which are removed from $[0,1]$ at the $n$-th step in the construction
of the Cantor ternary set $C$. Such sets are called Cantorvals (to emphasize
their similarity to unions of intervals and to the Cantor set
simultaneously). Formally, a \emph{Cantorval} (more precisely, an $\mathcal{M%
}$-Cantorval, see \cite{MO}) is a non-empty compact subset $S$ of the real
line such that $S$ is the closure of its interior, and both endpoints of any
non-degenerated component are accumulation points of one-point components of
$S$. A non-empty subset $C$ of the real line $\mathbb{R}$ will be called a
\emph{Cantor set} if it is compact, zero-dimensional, and has no isolated
points.

Let us observe that Theorem \ref{Guthrie Nymann} says, that $l_{1}$ can be
devided into 4 sets: $c_{00}$ and the sets connected with cases (1), (2) and
(3). Some algebraic and topological properties of these sets have been
recently considered in \cite{BBGS}.

We will describe sequences constructed by Weinstein and Shapiro, Ferens and
Guthrie and Nymann using the notion of multigeometric sequence. We call a
sequence \emph{multigeometric} if it is of the form%
\begin{equation*}
(k_{0},k_{1},\dots ,k_{m},k_{0}q,k_{1}q,\dots
,k_{m}q,k_{0}q^{2},k_{1}q^{2},\dots ,k_{m}q^{2},k_{0}q^{3}\dots )
\end{equation*}%
for some positive numbers $k_{0},\dots ,k_{m}$ and $q\in \left( 0,1\right) $%
.\ We will denote such a sequence by $(k_{0},k_{1},\dots ,k_{m};q)$. Keeping
in mind that the type of $E\left( x\right) $ is the same as $E\left( \alpha
x\right) $, for any $\alpha >0$, we can describe the Weinstein-Shapiro
sequence as%
\begin{equation*}
a=(8,7,6,5,4;\tfrac{1}{10}),
\end{equation*}%
the Ferens sequence as%
\begin{equation*}
b=(7,6,5,4,3;\tfrac{2}{27})
\end{equation*}%
and the Guthrie-Nymann sequence as%
\begin{equation*}
c=(3,2;\tfrac{1}{4}).
\end{equation*}%
Another interesting example of a sequence $d$ with $E(d)$ being Cantorval
was presented by R. Jones in (\cite{J}). The sequence is of the form%
\begin{equation*}
d=(3,2,2,2;\tfrac{19}{109}).
\end{equation*}%
In fact, Jones constructed continuum many sequences generating Cantorvals,
indexed by a parameter $q$, by proving that, for any positive number $q$
with
\begin{equation*}
\frac{1}{5}\leqslant \sum_{n=1}^{\infty }q^{n}<\frac{2}{9}
\end{equation*}%
(i.e. $\frac{1}{6}\leqslant q<\frac{2}{11}$) the achievement set of the
sequence
\begin{equation*}
(3,2,2,2;q)
\end{equation*}%
is a Cantorval.

The structure of the achievement sets $E(x)$ for multigeometric sequences $x$
was studied in the paper \cite{BFS}, which contains a necessary condition
for the achivement set $E(x)$\ to be an interval and sufficient conditions
for $E(x)$ to contain an interval or have Lebesgue measure zero. In the case
of a Guthrie-Nymann-Jones sequence
\begin{equation*}
x_{q}=(3,2,\dots ,2;q),
\end{equation*}%
of rank $m$ (i.e., with $m$ repeated 2's), the set $E(x_{q})\ $is an
interval if and only if $q\geqslant \frac{2}{2m+5}$, $E(x_{q})$ is a Cantor
set of measure zero if $q<\frac{1}{2m+2}$, and $E(x_{q})$ is a Cantorval if $%
q\in \{\frac{1}{2m+2}\}\cup \big[\frac{1}{2m},\frac{1}{2m+5}\big)$. In this
paper we reveal some structural properties of the sets $E(x_{q})$ for $q$
belonging to the \textquotedblleft misterious\textquotedblright\ interval $(%
\frac{1}{2m+2},\frac{1}{2m})$. In particular, we shall show that for almost
all $q$ in this interval the set $E(x_{q})$ has positive Lebesgue measure
and there is a decreasing sequence $(q_{n})$ convergent to $\frac{1}{2m+2}$
for which $E(x_{q_{n}})$ is a Cantor set of zero Lebesgue measure. The above
description of the structure of $E(x_{q})$ can be presented as follows:

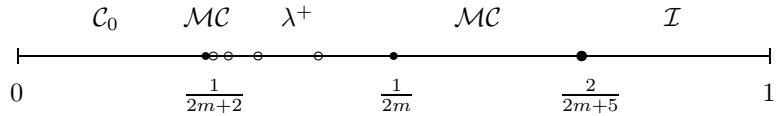
\setlength{\unitlength}{1mm}
\begin{picture}(96,25)(-25,0)
\put(3,12){\line(1,0){100}}
\put(2,6){0}
\put(13,16){$\mathcal{C}_0$}
\put(25,6){$\frac{1}{2m+2}$}
\put(25,16){$\MC$}
\put(38,16){$\lambda^+$}
\put(51,6){$\frac{1}{2m}$}
\put(61,16){$\MC$}
\put(75,6){$\frac{2}{2m+5}$}
\put(89,16){$\mathcal{I}$}
\put(102,6){$1$}
\put(3,11){\line(0,1){2}}
\put(103,11){\line(0,1){2}}
\put(28,12){\circle*{1}}
\put(29,12){\circle{1}}
\put(31,12){\circle{1}}
\put(35,12){\circle{1}}
\put(43,12){\circle{1}}
\put(53,12){\circle*{1}}
\put(78,12){\circle*{1.5}}
\end{picture}\newline
where $\mathcal{C}_{0}$ (resp. $\mathcal{MC}$, $\mathcal{I}$) indicates sets
of numbers $q$ for which the set $E(x_{q})$ is a Cantor set of zero Lebesgue
measure (resp. a Cantorval, an interval). The symbol $\lambda ^{+}$
indicates that for almost all $q$ in a given interval the sets $E(x_{q})$
have positive Lebesgue measure, which means that the set $Z=\{q\in \big(%
\frac{1}{2m+2},\frac{1}{2m}\big):\lambda (E(x_{q}))=0\}$ has Lebesgue
measure $\lambda (Z)=0$. Similar diagrams we use later in this paper.

The achievement sets of multigeometric sequences are partial cases of
self-similar sets of the form
\begin{equation*}
K(\Sigma;q)=\Big\{\sum_{n=0}^\infty a_nq^n:(a_n)_{n=0}^\infty\in\Sigma^\omega%
\Big\}
\end{equation*}
where $\Sigma\subset\mathbb{R}$ is a set of real numbers and $q\in(0,1)$.
The set $K(\Sigma;q)$ is self-similar in the sense that $K(\Sigma;q)=%
\Sigma+q\cdot K(\Sigma;q)$. Moreover, the set $K(\Sigma;q)$ can be found as
a unique compact solution $K\subset\mathbb{R}$ of the equation $K=\Sigma+qK$.

It follows that for a multigeometric sequence $x_{q}=(k_{0},\dots ,k_{m};q)$
the achievement set $E(x)$ coincides with the self-similar set $K(\Sigma ;q)$
for the set
\begin{equation*}
\Sigma =\Big\{\sum_{n=0}^{m}k_{n}\varepsilon _{n}:(\varepsilon
_{n})_{n=0}^{m}\in \{0,1\}^{m+1}\Big\}
\end{equation*}%
of all possible sums of the numbers $k_{0},\dots ,k_{m}$. This makes
possible to apply for studying the achievement sets $E(x_{q})$ the theory of
self-similar sets developed in \cite{H}, \cite{Schief} and, first of all, in
\cite{Fa}.

In this paper we shall describe some topological and measure properties of
the self-similar sets $K(\Sigma ;q)$ depending on the value of the
similarity ratio $q\in (0,1)$, and shall apply the obtained result to
establishing topological and measure properties of achievement sets of
multigeometric progressions. To formulate the principal results we need to
introduce some number characteristics of compact subsets $A\subset \mathbb{R}
$.

Given a compact subset $A\subset \mathbb{R}$ containing more than one point
let
\begin{equation*}
\mathrm{diam}\,A=\sup \{|a-b|:a,b\in A\}
\end{equation*}%
be the diameter of $A$ and
\begin{equation*}
\delta (A)=\inf \{|a-b|:a,b\in A,\;a\neq b\}\mbox{ and }\Delta (A)=\sup
\{|a-b|:a,b\in A,\;(a,b)\cap A=\emptyset \}
\end{equation*}%
be the smallest and largest gaps in $A$, respectively. Observe that $A$ is
an interval (equal to $[\min A,\max A]$) if and only if $\Delta (A)=0$.

Also put
\begin{equation*}
I(A)=\frac{\Delta (A)}{\Delta (A)+\mathrm{diam}\,A}\mbox{ \ \ and \ \ }%
i(A)=\inf \{I(B):B\subset A,\;\;2\leq |B|<\omega \}.
\end{equation*}%
In particular, given a finite subset $\Sigma \subset \mathbb{R}$ of
cardinality $|\Sigma |\geq 2$, we will write it as $\Sigma =\{\sigma
_{1},\dots ,\sigma _{s}\}$ for real numbers $\sigma _{1}<\dots <\sigma _{s}$%
. Then we have
\begin{equation*}
\mathrm{diam}(\Sigma )=\sigma _{s}-\sigma _{1},\;\;\delta (\Sigma
)=\min_{i<s}(\sigma _{i+1}-\sigma _{i}), \mbox{ \ and \ }\Delta (\Sigma
)=\max_{i<s}(\sigma _{i+1}-\sigma _{i}).
\end{equation*}

\begin{theorem}
\label{main} Let $\Sigma=\{\sigma_1,\dots,\sigma_s\}$ for some real numbers $%
\sigma_1<\dots<\sigma_s$. The self-similar sets $K(\Sigma;q)$ where $%
q\in(0,1)$ have the following properties:

\begin{enumerate}
\item $K(\Sigma;q)$ is an interval if and only if $q\ge I(\Sigma)$;

\item $K(\Sigma;q)$ is not a finite union of intervals if $q<I(\Sigma)$ and $%
\Delta(\Sigma)\in\{\sigma_2-\sigma_1,\sigma_s-\sigma_{s-1}\}$;

\item $K(\Sigma;q)$ contains an interval if $q\ge i(\Sigma)$;

\item If $d=\frac{\delta(\Sigma)}{\mathrm{diam}(\Sigma)}<\frac1{3+2\sqrt{2}}$
and $\frac1{|\Sigma|}<\frac{\sqrt{d}}{1+\sqrt{d}}$, then for almost all $q\in%
\big(\frac1{|\Sigma|},\frac{\sqrt{d}}{1+\sqrt{d}}\big)$ the set $K(\Sigma;q)$
has positive Lebesgue measure and the set $K(\Sigma;\sqrt{q})$ contains an
interval;

\item $K(\Sigma;q)$ is a Cantor set of zero Lebesgue measure if $%
q<\frac1{|\Sigma|}$ or, more generally, if $q^n<\frac1{|\Sigma_n|}$ for some
$n\in\mathbb{N}$ where $\Sigma_n=\big\{%
\sum_{k=0}^{n-1}a_kq^k:(a_k)_{k=0}^{n-1}\in\Sigma^n\big\}$.

\item If $\Sigma\supset\{a,a+1,b+1,c+1,b+|\Sigma|,c+|\Sigma|\}$ for some
real numbers $a,b,c\in\mathbb{R}$ with $b\ne c$, then there is a strictly
decreasing sequence $(q_n)_{n\in\omega}$ with $\lim_{n\to\infty}q_n=\frac1{|%
\Sigma|}$ such that the sets $K(\Sigma;q_n)$ has Lebesgue mesure zero.
\end{enumerate}
\end{theorem}

The statements (1)--(3) from this theorem will be proved in Section~\ref%
{s:int}, the statement (4) in Section~\ref{s:pos} and (5),(6) in Section~\ref%
{s:null}. Writing that for almost all $q$ in an interval $(a,b)$ some
property $\mathcal{P}(q)$ holds we have in mind that the set $Z=\{q\in (a,b):%
\mathcal{P}(q)$ does not hold$\}$ has Lebesgue measure $\lambda (Z)=0$.

\section{Intervals and Cantorvals}

\label{s:int}

In this section we generalize results of \cite{BFS} detecting the
self-similar sets $K(\Sigma ;q)$ which are intervals or Cantorvals. In the
following theorem we prove the statements (1)--(3) of Theorem~\ref{main}.

\begin{theorem}
\label{th3} Let $q\in(0,1)$ and $\Sigma=\{\sigma_1,\dots,\sigma_s\}\subset%
\mathbb{R}$ be a finite set with $\sigma_1<\dots<\sigma_s$. The self-similar
set $K(\Sigma;q)=\big\{\sum_{i=0}^{\infty }a
_{i}q^{i}:(a_{i})_{i\in\omega}\in \Sigma ^{\omega}\big\}$

\begin{enumerate}
\item is an interval if and only if $q\ge I(\Sigma)$;

\item contains an interval if $q\ge i(\Sigma)$;

\item is not a finite union of intervals if $q<I(\Sigma)$ and $%
\Delta(\Sigma)\in\{\sigma_2-\sigma_1,\sigma_s-\sigma_{s-1}\}$.
\end{enumerate}
\end{theorem}

\begin{proof}
1. Observe that $\mathrm{diam} K(\Sigma;q)=\mathrm{diam}(\Sigma)/(1-q)$.
Assuming that $q\geq I(\Sigma)=\Delta(\Sigma)/(\Delta(\Sigma)+\mathrm{diam}%
\Sigma)$, we conclude that $\Delta(\Sigma)\le q\cdot \mathrm{diam}
(\Sigma)/(1-q)=q\cdot\mathrm{diam} K(\Sigma;q)$, which implies that
\begin{equation*}
\Delta(K(\Sigma;q))=\Delta(\Sigma+q\cdot K(\Sigma;q))\le \Delta(q\cdot
K(\Sigma;q))=q\cdot \Delta(K,\Sigma;q).
\end{equation*}
Since $q<1$ this inequality is possible only in case $\Delta(K(\Sigma;q))=0$%
, which means that $K(\Sigma;q)$ is an interval.

If $q<\Delta(\Sigma)/(\Delta(\Sigma)+\mathrm{diam}\Sigma)$, then $%
\Delta(\Sigma)>q\cdot \mathrm{diam}(\Sigma)/(1-q)=q\cdot \mathrm{diam}
(K(\Sigma;q))$ and we can find two consequtive points $a<b$ in $\Sigma$ with
$b=a+\Delta(\Sigma)>a+\mathrm{diam}(q K(\Sigma;q))$ and conclude that $%
[a,b]\cap K(\Sigma;q)=[a,b]\cap(\Sigma+qK(\Sigma;q))\subset [a,a+\mathrm{diam%
}(q\,K(\Sigma;q))]\ne [a,b]$, so $K( \Sigma ;q) $ is not an interval.
\smallskip

2. Now assume that $q\ge i(\Sigma)$ and find a subset $B\subset \Sigma$ such
that $I(B)=i(\Sigma)<q$. By the preceding item, the self-similar set $%
K(B;q)=B+q K(B;q)$ is an interval. Consequently, $K(\Sigma;q)$ contains the
interval $K(B;q)$. \smallskip

3. Finally assume that $\Delta(\Sigma)=\sigma_2-\sigma_1$ and $q<I(\Sigma)$.
Since for every $a\in\Sigma$ we get $K(\Sigma-a;q)=K(\Sigma;q)-\frac{a}{1-q}$%
, we can replace $\Sigma$ by its shift and assume that $\sigma_1=0$ and
hence $\Delta(\Sigma)=\sigma_2-\sigma_1=\sigma_2$. It follows from $%
q<I(\Sigma)=\sigma_2/(\sigma_2+\mathrm{diam} \Sigma)$ that for any $j\in
\mathbb{N}$, the interval $\big( \sum_{n=j+1}^{\infty }q^{n}\sigma
_{s},q^{j}\sigma _{2}\big) $ is nonempty and disjoint from $K\left( \Sigma
;q\right) $. Hence, no interval of the form $\left[ 0,\varepsilon \right] $
is included in $K\left( \Sigma ;q\right) $. But $0\in K\left( \Sigma
;q\right) $, so $K\left( \Sigma ;q\right) $ is not a finite union of closed
intervals. By analogy we can consider the case $\Delta(\Sigma)=\sigma_s-%
\sigma_{s-1}$.
\end{proof}

In particular, Theorem \ref{th3} implies:

\begin{corollary}
For $\Sigma =\{0,1,2,\dots ,s-1\}$ the set $K\left( \Sigma ;q\right) $ is an
interval if and only if $q\geq I(\Sigma)=\frac{1}{|\Sigma|}$.
\end{corollary}

\begin{corollary}
If $\{k,k+1,\dots ,k+n-1\}\subset \Sigma $, then $i(\Sigma)\le \frac1n$ and
for every $q\ge \frac1n$ the set $K(\Sigma;q)$ contains an interval.
\end{corollary}

In particular, for the Guthrie-Nymann-Jones multigeometric sequence $%
x_{q}=(3,2,\dots ,2;q)$ of rank $m$ the sumset $\Sigma =\{0,2,\dots
,2m+1,2m+3\}$ has cardinality $|\Sigma |=2m+2$, $I(\Sigma )=\frac{\Delta
(\Sigma )}{\Delta (\Sigma )+\mathrm{diam}\Sigma }=\frac{2}{2m+5}$, $i(\Sigma
)=\min \big\{ \frac{1}{2m},\frac{2}{2m+5}\big\} $, and $d=\frac{\delta
(\Sigma )}{\mathrm{diam}(\Sigma )}=\frac{1}{2m+3}$. So, for $q\in \big[\frac{%
2}{2m+5},1\big)$ the set $E(x_{q})=K(\Sigma ;q)$ is an interval and for $%
q\in \big[\frac{1}{2m},\frac{2}{2m+5}\big)$ a Cantorval.

\section{Sets of positive measure}

\label{s:pos}

In this section we shall prove the statement (4) of Theorem~\ref{main}
detecting numbers $q$ for which the self-similar set $K(\Sigma;q)$ has
positive Lebesgue measure $\lambda(K(\Sigma;q))$. For this we shall apply
the deep results of Boris Solomyak \cite{S} related to the distribution of
the random series $\sum_{n=0}^{\infty }a_n\lambda ^{n},$ where the
coefficients $a_n\in\Sigma$ are chosen independently with probability $\frac{%
1}{|\Sigma|}$ each.

Given a finite subset $\Sigma\subset\mathbb{R}$ consider the number
\begin{equation*}
\alpha(\Sigma)=\inf\big\{x\in(0,1):\exists
(a_n)_{n\in\omega}\in(\Sigma-\Sigma)^\omega\setminus\{0\}^\omega%
\mbox{ such
that }\sum_{n=0}^\infty a_nx^n=0\mbox{ and }\sum_{n=1}^\infty na_nx^{n-1}=0%
\big\}.
\end{equation*}

The first part of the following theorem was proved by Solomyak in \cite[1.2]%
{S}:

\begin{theorem}
\label{soloma} Let $\Sigma\subset\mathbb{R}$ be a finite subset. If $%
\frac1{|\Sigma|}<\alpha(\Sigma)$, then for almost all $q$ in the interval $%
\big(\frac1{|\Sigma|},\alpha(\Sigma)\big)$ the self-similar set $K(\Sigma;q)$
has positive Lebesgue measure and the set $K(\Sigma;\sqrt{q})$ contains an
interval.
\end{theorem}

\begin{proof}
By Theorem 1.2 of \cite{S}, for almost all $q\in\big(\frac1{|\Sigma|},%
\alpha(\Sigma)\big)$ the self-similar set $K(\Sigma;q)$ has positive
Lebesgue measure. Since $K(\Sigma;\sqrt{q})=K(\Sigma;q)+\sqrt{q}\cdot
K(\Sigma;q)$, the set $K(\Sigma;q)$ contains an interval, being the sum of
two sets of positive Lebesque measure (according to the famous Steinhaus
Theorem \cite{St}).
\end{proof}

The definition of Solomyak's constant $\alpha(\Sigma)$ does not suggest any
efficient way of its calculation. In \cite{S} Solomyak found an efficient
lower bound on $\alpha(\Sigma)$ based on the notion of a $(*)$-function,
i.e., a function of the form
\begin{equation*}
g(x)=-\sum_{k=1}^{n-1}x^k+\gamma x^n+\sum_{k=n+1}^\infty x^k
\end{equation*}%
for some $n\in\mathbb{N}$ and $\gamma\in[-1,1]$. In Lemma~3.1 \cite{S}
Solomyak proved that every $(*)$-function $g(x)$ has a unique critical point
on $[0,1)$ at which $g$ takes its minimal value. Moreover, for every $d>0$
there is a unique $(*)$-function $g_d(x)$ such that $\min_{[0,1)}g_d=-d$.
The unique critical point $x_d\in g_d^{-1}(-d)\in[0,1)$ of $g_d$ will be
denoted by $\underline{\alpha}(d)$. The following lower bound on the number $%
\alpha(\Sigma)$ follows from Proposition 3.2 and inequality (15) in \cite{S}.

\begin{lemma}
For every finite set $\Sigma\subset\mathbb{R}$ of cardinality $|\Sigma|\ge 2$
we get
\begin{equation*}
\alpha(\Sigma)\ge\underline{\alpha}(d)\mbox{ \ where \ }d=\frac{%
\delta(\Sigma)}{\mathrm{diam}(\Sigma)}.
\end{equation*}
\end{lemma}

The function $\underline{\alpha}(d)$ can be calculated effectively (at least
for $d\le\frac12$).

\begin{lemma}
\label{bound} If $0<d\le\frac1{3+2\sqrt{2}}$, then
\begin{equation*}
\underline{\alpha}(d)=\frac{\sqrt{d}}{1+\sqrt{d}}.
\end{equation*}
\end{lemma}

\begin{proof}
Observe that the minimal value of the $(*)$-function $g(x)=-x+\sum_{k=2}^%
\infty x^k=-x+\frac{x^2}{1-x}$ is equal to $-\frac1{3+2\sqrt{2}}$, which
implies that for $d\in \big(0,\frac1{3+2\sqrt{2}}\big]$ the number $%
\underline{\alpha}(d)$ is equal to the critical point of the unique $(*)$%
-function $g(x)=\gamma x+\sum_{k=2}^\infty x^k=-1+(\gamma-1)x+\frac1{1-x}$
with $\min_{[0,1)}g=-d$. This $(*)$-function has derivative $%
g^{\prime}(x)=(\gamma-1)+\frac1{(1-x)^2}$. If $x$ is the critical point of $%
g $, then $1-\gamma=\frac1{(1-x)^2}$ and the equality
\begin{equation*}
d=-1+(\gamma-1)x+\frac1{1-x}=-1-\frac{x}{(1-x)^2}+\frac1{1-x}
\end{equation*}%
has the solution
\begin{equation*}
x=1-\frac1{1+\sqrt{d}}=\frac{\sqrt{d}}{1+\sqrt{d}}
\end{equation*}%
which is equal to $\underline{\alpha}(d)$.
\end{proof}

For $d>\frac1{3+2\sqrt{2}}$ the formula for $\underline{\alpha}(d)$ is more
complex.

\begin{lemma}
If $\frac1{3+2\sqrt{2}}\le d\le \frac12$, then the value
\begin{equation*}
\underline{\alpha}(d)=\frac{1+d}3+\frac{\sqrt[3]{2}\cdot R}{6}+\frac{%
2d^2-8d-1}{3\sqrt[3]{2}\cdot R}
\end{equation*}
where
\begin{equation*}
R=\sqrt[3]{4d^3-24d^2+21d-5+3\sqrt{3}\sqrt{1-8d^3+39d^2-6d}}
\end{equation*}
can be found as the unique real solution of the qubic equation
\begin{equation*}
2(x-1)^3+(4-2d)(x-1)^2+3(x-1)+1=0.
\end{equation*}
\end{lemma}

\begin{proof}
Since the minimal values of the $(*)$-functions $g_1(x)=-x+\sum_{k=2}^\infty
x^k$ and $g(x)=-x-x^2+\sum_{k=3}^\infty x^k$ are equal to $-\frac1{3+2\sqrt{2%
}}$ and $-\frac12$, respectively, for $d\in\big[\frac1{3+2\sqrt{2}},\frac12%
\big]$ the number $\underline{\alpha}(d)$ is equal to the critical point of
a unique $(*)$-function
\begin{equation*}
g(x)=-x+\gamma x^2+\sum_{k=3}^\infty x^k=-1-2x+(\gamma-1)x^2+\frac1{1-x}
\end{equation*}
with $\min_{[0,1)}g=-d$. At the critical point $x$ the derivative of $g$
equals zero:
\begin{equation*}
0=g^{\prime}(x)=-2+2(\gamma-1)x+\frac1{(1-x)^2}
\end{equation*}%
which implies that
\begin{equation*}
\gamma-1=\frac1{2x}\Big(2-\frac1{(1-x)^2}\Big)=\frac{2x^2-4x+1}{2x(1-x)^2}.
\end{equation*}
After substitution of $\gamma-1$ to the formula of the function $g(x)$, we
get
\begin{equation*}
-d=-1-2x-\frac{2x^3-4x^2+x}{2(1-x)^2}+\frac1{1-x}.
\end{equation*}
This equation is equivalent to the qubic equation
\begin{equation*}
2(x-1)^3+(4-2d)(x-1)^2+3(x-1)+1=0.
\end{equation*}%
Solving this equation with the Cardano formulas we can get the solution $%
\underline{\alpha}(d)$ written in the lemma.
\end{proof}

\begin{remark}
Calculating the value $\underline{\alpha}(d)$ for some concrete numbers $d$,
we get
\begin{equation*}
\underline{\alpha}(\tfrac15)\approx 0.32482,\;\;\underline{\alpha}(\tfrac
14)\approx 0.37097, \; \underline{\alpha}(\tfrac13)\approx 0.42773,\;%
\underline{\alpha}(\tfrac12)=0.5 .
\end{equation*}
\end{remark}

Theorem~\ref{soloma} and Lemma~\ref{bound} imply:

\begin{corollary}
\label{cor8} Let $\Sigma\subset\mathbb{R}$ be a finite subset containing
more than three points and $d=\delta(\Sigma)/\mathrm{diam}(\Sigma)$. If $%
d\le \frac1{3+2\sqrt{2}}$ and $\frac{\sqrt{d}}{1+\sqrt{d}}>\frac1{|\Sigma|}$%
, then for almost all $q$ in the interval $\big(\frac1{|\Sigma|},\frac{\sqrt{%
d}}{1+\sqrt{d}}\big)$ the self-similar set $K(\Sigma;q)$ has positive
Lebesgue measure and the set $K(\Sigma;\sqrt{q})$ contains an interval.
\end{corollary}

\begin{remark}
Theorem~\ref{th3} says that for $q\in \lbrack i(\Sigma ),1)$ the set $%
K(\Sigma ;q)$ contains an interval. By Theorem~\ref{soloma} under certain
conditions the same is true for almost all $q\in \big[\frac{1}{\sqrt{|\Sigma
|}},\sqrt{\alpha (\Sigma )}\big)$. Let us remark that the numbers $i(\Sigma
) $ and $\frac{1}{\sqrt{|\Sigma |}}$ are incomparable in general. Indeed,
for the multigeometric sequence $(1,\dots ,1;q)$ containing $k>1$ units the
set $\Sigma =\{0,\dots ,k\}$ has
\begin{equation*}
i(\Sigma )=I(\Sigma )=\frac{1}{k+1}=\frac{1}{|\Sigma |}<\frac{1}{\sqrt{%
|\Sigma |}}.
\end{equation*}%
On the other hand, for the multigeometric sequence $(3^{k-1},3^{k-2},\dots
,3,1;q)$ the set $\Sigma =\{\sum_{n=0}^{k-1}3^{n}\varepsilon
_{n}:(\varepsilon _{n})_{n<k}\in \{0,1\}^{k}\}$ has cardinality $|\Sigma
|=2^{k}$, diameter $\mathrm{diam}(\Sigma )=(3^{k}-1)/2$, $d=\frac{\delta
(\Sigma )}{\mathrm{diam}(\Sigma )}=\frac{2}{3^{k}-1}$ and $i(\Sigma
)=I(\Sigma )=\frac{1}{4}+\frac{1}{4\cdot 3^{k-1}}>\frac{1}{\sqrt{2}^{k}}=%
\frac{1}{\sqrt{|\Sigma |}}$. Corollary~\ref{cor8} guarantees that for almost
all $q\in \big(\frac{1}{\sqrt{2}^{k}},\frac{\sqrt[4]{d}}{\sqrt{1+\sqrt{d}}}%
\big)$ the set $K(\Sigma ;q)$ contains an interval.
\end{remark}

Multigeometric sequences of the form%
\begin{equation*}
\left( k+m,\dots ,k+1,k;q\right)
\end{equation*}%
with $m\geq k$ we will call, after \cite{P-W}, \emph{Ferens-like sequences}.
The achievement set $E\left( x\right) $\ for a Ferens-like sequence
coincides with the self-similar set $K(\Sigma ;q)$\ for the set
\begin{equation*}
\Sigma =\{0,k,k+1,\dots ,n-k,n\}\text{.}
\end{equation*}%
where $n=(m+1)(2k+m)/2$.\ Sets $K\left( \Sigma ;q\right) $ with $\Sigma $ of
this form will be called \emph{Ferens-like fractals}.

Note that Guthrie-Nymann-Jones sequence of rank $m$ generates a Ferens-like
fractal (with $\Sigma =\{0,2,3,\dots ,2m+1,2m+3\}$. There are also
Ferens-like fractals which are not originated by any multigeometric sequence
(for example $K(\Sigma ;q)$ with $\Sigma =\left\{ 0,4,5,6,7,11\right\} $).
However, as an easy consequence of the main theorem of \cite{NS}, we obtain
for Ferens-like fractals \textquotedblleft trichotomy" analogous to that
formulated in Theorem \ref{Guthrie Nymann}. Moreover, some theorems
formulated for multigeometric sequences are in fact proved for $K(\Sigma ;q)$
(see for example Theorem 2 in \cite{BFS}).

\begin{example}
\label{ex9} For the Ferens-like sequence $x_{q}=(4,3,2;q)$ we get $\Sigma
=\{0,2,3,4,5,6,7,9\}$,
\begin{equation*}
d=\frac{\delta (\Sigma )}{\mathrm{diam}(\Sigma )}=\frac{1}{9}<\frac{1}{3+2%
\sqrt{2}}\mbox{ \ and \ }\frac{\sqrt{d}}{1+\sqrt{d}}=\frac{1}{4}>\frac{1}{6}%
=i(\Sigma ).
\end{equation*}%
By Corollary~\ref{cor8} (and Theorem~\ref{th3}), for almost all numbers $%
q\in \big(\frac{1}{8},1\big)$ the achievement set $E(x_{q})=K(\Sigma ;q)$
has positive Lebesgue measure (for $q<\frac{2}{11}=I(\Sigma )$ it is not a
finite union of intervals). By Theorem~\ref{th3}, for any $q\in \lbrack
i(\Sigma ),I(\Sigma ))=[\frac{1}{6};\frac{2}{11})$ the set $K(\Sigma ;q)$ is
a Cantorval. The structure of the sets $E(x_{q})=K(\Sigma ;q)$ is described
in the diagram:

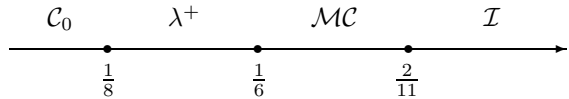
\setlength{\unitlength}{1mm}
\begin{picture}(56,23)(-30,0)
\put(0,12){\vector(1,0){74}}
\put(5,15){$\mathcal{C}_0$}
\put(12,7){$\frac{1}{8}$}
\put(21,15){$\lambda^+$}
\put(32,7){$\frac16$}
\put(40,15){$\MC$}
\put(51,7){$\frac2{11}$}
\put(63,15){$\mathcal{I}$}
\multiput(13,12)(20,0){3}{\circle*{1}}
\end{picture}\smallskip

More generally, for any Ferens-like fractal, $|\Sigma |=n-2k+3$, $\Delta
(\Sigma )=k$, $\delta \left( \Sigma \right) =1$, $I(\Sigma )=\frac{k}{n+k}$,
$i(\Sigma )=\min \big(\frac{1}{\left\vert \Sigma \right\vert -2},I(\Sigma )%
\big)$ and $d=\frac{1}{n}$. Moreover, if $n\geq 7$ then $\underline{\alpha }%
(d)=\frac{1}{\sqrt{n}+1}$. Therefore, one can check that for any Ferens-like
sequence we have $\underline{\alpha }(d)>i(\Sigma )$, and we can draw an
analogous diagram. The same result we can obtain for any Ferens-like fractal
with $k=2$ (even if it is not originated by any Ferens-like sequence).
However, there are Ferens-like fractals with $\underline{\alpha }%
(d)<i(\Sigma )$ (for example $K(\Sigma ;q)$ with $\Sigma =\left\{
0,3,4,7\right\} $ or $\Sigma =\left\{ 0,4,5,6,7,11\right\} $).
\end{example}

\begin{example}
\label{ex9a} For the Guthrie-Nymann-Jones sequence $x_{q}=(3,2,\dots ,2;q)$
of rank $m\geq 2$ we get $\Sigma =\{0,2,3,\dots ,2m+1,2m+3\}$, $|\Sigma
|=2m+2$, $I(\Sigma )=\frac{2}{2m+5}$, $i(\Sigma)=\min \big\{\frac{1}{2m},%
\frac{2}{2m+5}\big\}$, $d=\frac{1}{2m+3}$ and $\underline{\alpha }(d)=1/(1+%
\sqrt{2m+3})$. Moreover, we have $d<\frac{1}{3+2\sqrt{2}}$ and $\underline{%
\alpha }(d)\geq i(\Sigma )>\frac{1}{2m+2}=\frac{1}{|\Sigma |}$. So, we can
apply Corollary~\ref{cor8} and conclude that for almost all numbers $q\in %
\big(\frac{1}{2m+2},\frac{1}{2m}\big)$ the self-similar set $K(\Sigma ;q)$
has positive Lebesgue measure. By Theorem~\ref{th3}, for any $q\in \lbrack
i(\Sigma ),\frac{2}{2m+5})$ the set $K(\Sigma;q)$ is a Cantorval and for all
$q\in \lbrack \frac{2}{2m+5},1)$ it is an interval.\newline
For $m=1$ we obtain $\underline{\alpha }(d)=\underline{\alpha }(\frac{1}{5})>%
\frac{2}{7}$. Therefore, for almost all numbers $q\in \big(\frac{1}{4},\frac{%
2}{7}\big)$ the set $K(\Sigma ;q)$ has positive Lebesgue measure.
\end{example}

\section{Self-similar sets of zero Lebesgue measure}

\label{s:null}

The results of the preceding section yields conditions under which for
almost all $q$ in an interval $\big[\frac{1}{|\Sigma|},\alpha(\Sigma)\big)$
the set $K( \Sigma ;q) $ has positive Lebesgue measure. In this section we
shall show that this interval can contain infinitely many numbers $q$ with $%
\lambda (K(\Sigma ;q)) =0$ thus proving the statements (5) and (6) of
Theorem~\ref{main}.

\begin{theorem}
\label{th9}If there exists $n\in \mathbb{N}$ such that
\begin{equation*}
\Big\vert \sum_{i=0}^{n-1}q^{i}\Sigma \Big\vert \cdot q^{n}<1
\end{equation*}
then the set $K(\Sigma ,q)$ has measure zero.
\end{theorem}

\begin{proof}
Denote $K:=K(\Sigma ,q)$. From the equality $K=\Sigma +qK$ we obtain, by
induction, that
\begin{equation*}
K=\sum_{i=0}^{n-1}q^{i}\Sigma +q^{n}K.
\end{equation*}%
Let $\Sigma _{n}=\sum_{i=0}^{n-1}q^{i}\Sigma $. If $|\Sigma _{n}|\cdot
q^{n}<1$, then
\begin{equation*}
\lambda (K)\leq |\Sigma _{n}|\cdot q^{n}\cdot \lambda (K)<1\cdot \lambda (K)
\end{equation*}%
which is possible only if $\lambda (K)=0$.
\end{proof}

To use the latter theorem we need a technical lemma:

\begin{lemma}
\label{l2}For any integer numbers $s>1$ and $n>1$ the unique positive
solution $q$ of the equation%
\begin{equation}
x+x^{2}+\dots+x^{n-1}=\frac{1}{s-1}  \label{1}
\end{equation}%
is greater than $\frac{1}{s}$. Moreover, there is $n_{0}\in\mathbb{N}$ such
that for any $n>n_{0}$\
\begin{equation}
\left( s^{n}-2^{n-1}\right) \cdot q^{n}<1\text{.}  \label{2}
\end{equation}

\begin{proof}
Clearly
\begin{equation*}
\sum_{i=1}^{n-1}\left( \frac{1}{s}\right) ^{i}=\frac{1}{s-1}\cdot \left( 1-%
\frac{1}{s^{n-1}}\right) <\frac{1}{s-1}\text{,}
\end{equation*}%
so $q>\frac{1}{s}$. From the equality
\begin{equation*}
\frac{1}{s-1}=\sum_{i=1}^{n-2}\left( \frac{1}{s}\right) ^{i}+\frac{1}{\left(
s-1\right) s^{n-2}}
\end{equation*}%
we obtain%
\begin{equation*}
q^{n-1}=\frac{1}{s-1}-\sum_{i=1}^{n-2}q^{i}<\frac{1}{s-1}-\sum_{i=1}^{n-2}%
\left( \frac{1}{s}\right) ^{i}=\frac{1}{\left( s-1\right) s^{n-2}}\text{.}
\end{equation*}%
Using the latter inequality and the equality%
\begin{equation*}
\frac{1}{s-1}=\frac{q-q^{n}}{1-q}
\end{equation*}%
we have%
\begin{equation*}
\frac{1-q}{s-1}=q\left( 1-q^{n-1}\right) >q\left( 1-\frac{1}{\left(
s-1\right) s^{n-2}}\right) \text{.}
\end{equation*}%
Therefore,%
\begin{equation*}
1-q>\left( s-1\right) q-\frac{q}{s^{n-2}}
\end{equation*}%
(which means that $sq-\frac{q}{s^{n-2}}<1$)\ and finally%
\begin{equation}
q<\frac{1}{s\left( 1-\frac{1}{s^{n-1}}\right) }\text{.}  \label{3}
\end{equation}%
From Bernoulli's inequality it follows that%
\begin{equation*}
\left( 1-\frac{1}{s^{n-1}}\right) ^{n}\geq 1-\frac{n}{s^{n-1}}
\end{equation*}%
and, by (\ref{3}), we have%
\begin{equation*}
q^{n}<\frac{1}{s^{n}\cdot \left( 1-\frac{n}{s^{n-1}}\right) }\text{.}
\end{equation*}%
Consequently,%
\begin{equation*}
\left( s^{n}-2^{n-1}\right) \cdot q^{n}<\frac{s^{n}\cdot \left( 1-\frac{%
2^{n-1}}{s^{n}}\right) }{s^{n}\cdot \left( 1-\frac{n}{s^{n-1}}\right) }
\end{equation*}%
Obviously, for $n$ greater then some $n_{0}$\
\begin{equation*}
\frac{2^{n-1}}{s}>n
\end{equation*}%
and hence%
\begin{equation*}
\frac{2^{n-1}}{s^{n}}>\frac{n}{s^{n-1}}
\end{equation*}%
which proves (\ref{2}).
\end{proof}
\end{lemma}

\begin{theorem}
\label{Th2}If a finite subset $\Sigma\subset\mathbb{R}$ contains the set $%
\{a,a+1,b+1,c+1,b+|\Sigma|,c+|\Sigma|\}$ for some real numbers $a,b,c$ with $%
b\ne c$, then there is a decreasing sequence $(q_{n})_{n=1}^{\infty }$
tending to $\frac{1}{|\Sigma|}$\ such that, for any $n\in\mathbb{N}$, the
self-similar set $K(\Sigma,q_{n})$ has Lebesgue measure zero.
\end{theorem}

\begin{proof}
Let $s=|\Sigma|$ and for every $n$ denote by $q_{n}$ the unique positive
solution of the equation (\ref{1}) from Lemma \ref{l2}. Let $n_{0}$ be a
natural number such that%
\begin{equation*}
\left( s^{n}-2^{n-1}\right) \cdot \left( q_{n}\right) ^{n}<1
\end{equation*}%
for any $n>n_{0}$. Clearly $(q_{n}) _{n=n_{0}}^{\infty }$ is a decreasing
sequence and $\lim_{n\rightarrow \infty }q_{n}=\frac{1}{s}$. It suffices to
show that $K(\Sigma ,q)$ has measure zero for $n>n_{0}$.\newline
Taking into account that each $q_{n}$ is a solution of (\ref{1}), we
conclude that
\begin{equation*}
a+\sum_{i=1}^{n-1}(s-1+\varepsilon _{i})( q_{n})
^{i}=(a+1)+\sum_{i=1}^{n-1}\varepsilon _{i}( q_{n}) ^{i}
\end{equation*}%
for any $\varepsilon _{i}\in \{b+1,c+1\}\subset\Sigma$. Therefore%
\begin{equation*}
\left\vert \sum_{i=1}^{n-1}\left( q_{n}\right) ^{i}\Sigma \right\vert \leq
s^{n}-2^{n-1}\text{.}
\end{equation*}%
Hence, by Lemma \ref{l2},
\begin{equation*}
\left\vert \sum_{i=1}^{n-1}\left( q_{n}\right) ^{i}\Sigma \right\vert \cdot
\left( q_{n}\right) ^{n}<1.
\end{equation*}%
and we can apply Theorem \ref{th9} to conclude that $K(\Sigma ,q)$ has
Lebesgue measure zero.
\end{proof}

The condition
\begin{equation}
\big\{a,a+1,b+1,c+1,b+|\Sigma|,c+|\Sigma|\big\} \subset \Sigma  \tag{$*$}
\label{gw}
\end{equation}%
looks a bit artificial but it can be easily verified for many sumsets $%
\Sigma $ of multigeometric sequences.

In particular, for the Guthrie-Nymann-Jones sequence of rank $m\geq 1$
\begin{equation*}
x_{q}=(3,2,\dots ,2;q),
\end{equation*}%
the sumset $\Sigma =\{0,2,3,\dots ,2m+1,2m+3\}$ has cardinality $|\Sigma
|=2m+2$. Observe that for the set $\Sigma $ the condition $(\ast )$ holds
for $a=2$, $b=1$ and $c=-1$. Because of that Theorem~\ref{Th2} yields a
sequence $(q_{n})_{n=1}^{\infty }\searrow \frac{1}{2m+2}$ such that for
every $n\in \mathbb{N}$ the self-similar set $E(x_{q_{n}})$ is a Cantor sets
of zero Lebesgue measure.

By \cite{BFS}, for $q=\frac{1}{2m+2}$ the achievement set $E(x_{q})$ is a
Cantorval. Therefore, if $m>2$, there are three ratios $p<q<r$ such that $%
E(x_{p})$ and $E(x_{r})$ are Cantor sets while $E(x_{q})$ is a Cantorval. By
our best knowledge it is the first result of this type for multigeometric
sequences.

Now we will focus on Ferens-like sequences $x_{q}=(m+k,\dots ,k;q)$ where $%
m\geq k$. \smallskip

For $k=1$ the Ferens-like sequence $x_{q}=(m+1,\dots ,2,1;q)$ has%
\begin{equation*}
\Sigma =\big\{0,1,2,\dots ,(m+2)\left( m+1\right) /2\big\}\text{.}
\end{equation*}%
The set $E(x_{q})$ is a Cantor set (for $q<\frac{1}{|\Sigma |}$) or an
interval (for $q\geq \frac{1}{|\Sigma |}$); see Theorem 7 in \cite{BFS}),
Theorem \ref{kakeya} or Theorem~\ref{th3}. \smallskip

For $k=2$, the \textquotedblleft shortest" Ferens-like sequence is $%
x_{q}=(4,3,2;q)$. For this sequence%
\begin{equation*}
\Sigma =\left\{ 0,2,3,4,5,6,7,9\right\} \text{.}
\end{equation*}%
Note that the same $\Sigma $ has Guthrie-Nymann-Jones sequence $(3,2,2,2;q)$
(see Example \ref{ex9a}). It follows that $E(x_{q})$ is a Cantor set for $%
q\in \big(0,\frac{1}{8}\big)$ and $E(x_{q})$ is a Cantorval for $q=\frac{1}{8%
}$. By Theorem~\ref{th3}, $K(\Sigma ;q)$ is an interval for $q\geq I(\Sigma
)=\frac{2}{11}$ and a Cantorval for $q\in \big(\frac{1}{6},\frac{2}{11}\big)$%
. As shown in Example~\ref{ex9a}, for almost all $q\in \big(\frac{1}{8},%
\frac{1}{6}\big)$ the set $K(\Sigma ;q)$ has positive Lebesgue measure.
Using Theorem~\ref{Th2}, we can find a decreasing sequence $(q_{n})$ tending
to $\frac{1}{8}$ for which the sets $K(\Sigma ;q_{n})$ have zero Lebesgue
measure.

\smallskip

For $k=3$ the \textquotedblleft shortest" Ferens-like sequence is $%
x_{q}=(6,5,4,3;q)$. For this sequence%
\begin{equation*}
\Sigma =\left\{ 0,3,\dots ,15,18\right\}
\end{equation*}%
and $|\Sigma |=15$. Since $1\in \frac{1}{15}\Sigma $ the set $\Sigma
_{2}=\Sigma +\frac{1}{15}\Sigma $ has less than $\left\vert 15\right\vert
^{2}$ elements (for example $4$ can be presented as $4+0$ or as $3+1$).
Therefore $\frac{1}{15^{2}}|\Sigma _{2}|<1$ and for $q=\frac{1}{15}$ the set
$E(x_{q})$ is a Cantor set according to Theorem \ref{th9}. Moreover,
calculating for $q=\frac{1}{14}>\frac{1}{15}$ the cardinality
\begin{equation*}
|\Sigma_3|=|\Sigma +q\Sigma +q^{2}\Sigma |=2655<14^{3}
\end{equation*}%
and applying Theorem \ref{th9}, we conclude that the achievement set $%
E(x_{q})$ is a Cantor set of zero Lebesgue measure for $q=\frac{1}{14}$. On
the other hand, Corollary~\ref{cor8} implies that for almost all $q\in \big(%
\frac{1}{15},\frac{1}{1+\sqrt{18}})$ the achievement set $E(x_{q})$ has
positive Lebesque measure. The set $\Sigma $ has $i(\Sigma )=\frac{1}{13}$
and $I(\Sigma )=\frac{3}{21}=\frac{1}{7}$. So, in this case we have the
diagram:

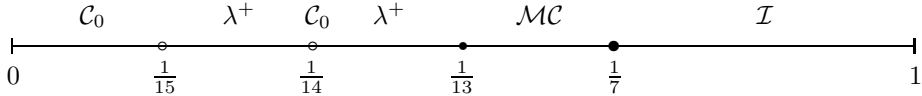
\setlength{\unitlength}{1mm}
\begin{picture}(56,23)(-20,0)
\put(-7,12){\line(1,0){120}}
\put(2,15){$\mathcal{C}_0$}
\put(11.5,7){$\frac{1}{15}$}
\put(21,15){$\lambda^+$}
\put(41,15){$\lambda^+$}
\put(32,15){$\mathcal{C}_0$}
\put(31,7){$\frac{1}{14}$}
\put(51,7){$\frac{1}{13}$}
\put(60,15){$\MC$}
\put(72,7){$\frac17$}
\put(92,15){$\mathcal{I}$}
\put(13,12){\circle{1}}
\put(-7.7,7){$0$}
\put(33,12){\circle{1}}
\put(53,12){\circle*{1}}
\put(73,12){\circle*{1.5}}
\put(-7,11){\line(0,1){2}}
\put(113,11){\line(0,1){2}}
\put(112.3,7){$1$}
\end{picture}\newline
As in the previous case, we can use Theorem~\ref{Th2} (taking $a=b=3$ and $%
c=-1$) and find a decreasing sequence $(q_{n})$ tending to $\frac{1}{15}$
such that all $E(x_{q_{n}})$ have zero Lebesgue measure.

\smallskip

Suppose now that $k>3$. For the Ferens-like sequence $x_{q}=(k+m,\dots
,k+1,k;q)$ its sumset $\Sigma $ contains the number $|\Sigma |$, which
implies that $|\Sigma +q\Sigma |<|\Sigma |^{2}$ for $q=\frac{1}{|\Sigma |}$
and therefore $E(x_{q})$ is a Cantor set of zero measure according to
Theorem~\ref{th9}.

\section{Rational ratios}

For a contraction ratio $q\in\{\frac1{n+1}:n\in\mathbb{N}\}$ self-similar
sets of positive Lebesgue measure can be characterized as follows:

\begin{theorem}
\label{t12} Let $\Sigma \subset \mathbb{Z}$ be a finite set, $%
q\in\{\frac1{n+1}:n\in\mathbb{N}\}$ and $\Sigma
_{n}=\sum_{i=0}^{n-1}q^{i}\Sigma $ for $n\in\mathbb{N}$. For the compact set
$K=K(\Sigma;q)$ the following conditions are equivalent:

\begin{itemize}
\item[(i)] $|\Sigma_n| \cdot q^n \geq 1$ for all $n\in\mathbb{N}$;

\item[(ii)] $\inf_{n\in\mathbb{N}}|\Sigma_{n}|\cdot q^n>0$,

\item[(iii)] $\lambda (K)>0.$
\end{itemize}
\end{theorem}

\begin{proof}
The implication (iii)$\Rightarrow $(i) follows from Theorem \ref{th9} while
(i)$\Rightarrow $(ii) is trivial. It remains to prove (ii)$\Rightarrow $%
(iii). Suppose that $\lambda (K)=0$. Given any $r>0$ consider the $r$%
-neighborhood $H(K,r) =\{ h\in \mathbb{R}:\mathrm{dist}(h,K) <r\}$ of the
set $K=K(\Sigma;q)$. Take any point $z\in \big\{\sum_{i=n}^\infty
x_iq^i:\forall i\ge n\;x_i\in\Sigma\big\}$ and observe that $%
\Sigma_n+z\subset K=\big\{\sum_{i=0}^\infty
x_iq^i:(x_i)_{i\in\omega}\in\Sigma^\omega\big\}$, which implies that $%
H(\Sigma _{n}+z,r) \subset H(K,r)$ for all $r>0$. The continuity of the
Lebesgue measure implies that $\lambda(H(K,r))\rightarrow 0$ when $r$ tends
to zero. It follows from $\Sigma\subset\mathbb{Z}$ and $\frac1q\in\mathbb{N}$
that
\begin{equation*}
\Sigma _{n}\subset q^{n-1}\cdot \mathbb{Z}\text{. }
\end{equation*}%
Hence, for any two different points $x$ and $y$ from $\Sigma _{n}$, the
distance between $x$ and $y$ is no less then $q^{n-1}>q^n$. Therefore, for
any $n\in \mathbb{N}$,%
\begin{equation*}
|\Sigma_n|\cdot q^n=\lambda \big( H\big( \Sigma_{n},\tfrac{1}{2}q^{n}\big)%
\big)=\lambda\big(H\big(\Sigma_n+z,\tfrac12q^n\big)\big)\le\lambda(K,%
\tfrac12q^n)
\end{equation*}%
which means that $\lim_{n\rightarrow \infty }|\Sigma _{n}|\cdot q^n=0$.
\end{proof}

Theorems~\ref{t12} combined with Corollary 2.3 of \cite{Schief} imply the
following corollary.

\begin{corollary}
For a finite subset $\Sigma\subset\mathbb{Z}$ and the number $q=\frac{1}{%
|\Sigma|}<1$ the following conditions are equivalent:

\begin{itemize}
\item[(1)] $K(\Sigma;q)$ has positive Lebesgue measure;

\item[(2)] $K(\Sigma;q)$ contains an interval;

\item[(3)] for every $n\in\mathbb{N}$ the set $\Sigma
_{n}=\sum_{k=0}^{n-1}q^k\Sigma$ has cardinality $|\Sigma _{n}|=|\Sigma|^{n}$.
\end{itemize}
\end{corollary}

\begin{problem}
Is it true that for a finite set $\Sigma\subset\mathbb{Z}$ and any
(rational) $q\in(0,1)$ the self-similar set $K(\Sigma;q)$ has positive
Lebesgue measure if and only if it contains an interval?
\end{problem}

\begin{remark}
According to \cite{Ex}, there exists a 10-element set $\Sigma $ on the
complex plane $\mathbb{C}$ such that for $q=\frac{1}{3}$ the self-similar
compact set $K(\Sigma ;q)=\Sigma +qK(\Sigma ;q)\subset \mathbb{C}$ has
positive Lebesgue measure and empty interior in $\mathbb{C}$.
\end{remark}

\end{document}